\documentclass[11pt]{conm-p-l}
\usepackage{aliascnt}

\usepackage[margin=1in]{geometry}
\usepackage{amsmath}
\usepackage{amssymb}
\usepackage{hyperref}
\usepackage{amsthm}
\usepackage{cleveref}

\usepackage{amssymb}
\newcommand{\dashedrightarrowhead}{\mathrel{\text{\setbox0\hbox{$\twoheadrightarrow$}\rlap{\hbox to \wd0{\hfil$\dashrightarrow$}}\box0}}}

\usepackage{todonotes}

\DeclareMathOperator{\Image}{Im}
\DeclareMathOperator{\height}{ht}
\def\P {\mathbb{P}}
\newcommand{\CC}{\mathbb{C}}

\DeclareMathOperator{\spr}{spr}

\theoremstyle{definition}
\newtheorem{theorem}{Theorem}[section]

\newaliascnt{definition}{theorem}
\newtheorem{definition}[definition]{Definition}
\aliascntresetthe{definition}
\crefname{definition}{definition}{definitions}
\Crefname{definition}{Definition}{Definitions}

\newaliascnt{remark}{theorem}
\newtheorem{remark}[remark]{Remark}
\aliascntresetthe{remark}
\crefname{remark}{remark}{remarks}
\Crefname{remark}{Remark}{Remarks}

\newaliascnt{lemma}{theorem}
\newtheorem{lemma}[lemma]{Lemma}
\aliascntresetthe{lemma}
\crefname{lemma}{lemma}{lemmas}
\Crefname{lemma}{Lemma}{Lemmas}

\newaliascnt{proposition}{theorem}
\newtheorem{proposition}[proposition]{Proposition}
\aliascntresetthe{proposition}
\crefname{proposition}{proposition}{propositions}
\Crefname{proposition}{Proposition}{Propositions}

\newaliascnt{corollary}{theorem}
\newtheorem{corollary}[corollary]{Corollary}
\aliascntresetthe{corollary}
\crefname{corollary}{corollary}{corollaries}
\Crefname{corollary}{Corollary}{Corollaries}

\newaliascnt{example}{theorem}
\newtheorem{example}[example]{Example}
\aliascntresetthe{example}
\crefname{example}{example}{examples}
\Crefname{example}{Example}{Examples}

\newaliascnt{problem}{theorem}
\newtheorem{problem}[problem]{Problem}
\aliascntresetthe{problem}
\crefname{problem}{problem}{problems}
\Crefname{problem}{Problem}{Problems}

\newaliascnt{solution}{theorem}

\aliascntresetthe{solution}
\crefname{solution}{solution}{solutions}
\Crefname{solution}{Solution}{Solutions}
\DeclareMathOperator{\Seg}{Seg}

\title{Segre-Determinantal Loci and the Image Variety for Three Flatland Cameras}

\author[Alstad]{Colin Alstad}
\email{calstad@clemson.edu}

\author[Duff]{Timothy Duff}
\email{tduff@missouri.edu}

\author[Katzman]{Mordechai Katzman}
\email{m.katzman@sheffield.ac.uk}

\begin{document}

\begin{abstract}
  Motivated by applications of algebraic geometry to reconstruction problems in computer vision, we initiate a study of the equations of degeneracy loci associated with linearly dependent points on Segre varieties. When these points are constrained to lie on a common hyperplane, we prove that the vanishing ideals of these loci are prime, Cohen-Macaulay, and generated by the natural maximal minors, and that these minors form a universal Gr\"{o}bner basis.
\end{abstract}

\maketitle
\section{Introduction} \label{sec:intro}

The story of this paper begins with a problem posed at the Fields Institute's Summer 2025 Workshop on Applications of Commutative Algebra.
\begin{problem}\label{problem:flatland}
\emph{Given $n$ triples of points on the projective line, 
$$(p_{1j}, p_{2j}, p_{3j}) \in \left( \mathbb{P}^{1}\right)^{\times 3}  ,\quad 1\le j\le n,$$ 
when does there exist a 2-dimensional reconstruction consisting of points $$q_{1},\dots,q_{n}\in\mathbb{P}^{2}$$ and three full-rank linear projections  $$A_{i}:\mathbb{P}^{2}\dashrightarrow\mathbb{P}^{1}, \quad 1\le i \le 3,$$
such that $A_i q_j \sim p_{ij}$ for all $1\le i\le 3$ and $1\le j\le n$?}
\end{problem}
The entities $ A_i, q_j$ and $p_{ij}$ appearing in \Cref{problem:flatland}  may be represented by matrices with dimensions $2\times 3$, $3\times 1$, and $2\times 1$, respectively, and the predicate $\sim $ denotes equality up to scale.
The linear projections $A_1, A_2, A_3$ are commonly known as \emph{flatland cameras}, and play an important role in an emerging interdisciplinary field known as \emph{algebraic vision}~\cite{kileel2022snapshot}.

\Cref{problem:flatland} concerns reconstructing a 2D scene from 1D images.
Higher-dimensional analogues are naturally also of interest: in particular, reconstructing a 3D scene from 2D images, where as few as two cameras may be used.
Let us point out, however, that the reconstruction problem for just two flatland cameras is completely trivial!
This can be seen through the following construction, adapted from~\cite{agarwal2026computer}: if we choose coordinates in $\mathbb{P}^1$ so that
$$
p_{1j}\sim \begin{pmatrix}
a_j\\
    1
\end{pmatrix}, \, \,
p_{2j} \sim 
\begin{pmatrix}
    1\\
    b_j
\end{pmatrix} \in \mathbb{P}^1,
\quad 
1\le j \le n,
$$
then one may easily produce a reconstruction by setting
$$
A_1 = \begin{pmatrix}
1 & 0 & 0 \\
0 & 1 & 0 
\end{pmatrix},
\quad 
A_2 = \begin{pmatrix}
0 & 1 & 0  \\
0 & 0 & 1 
\end{pmatrix}, 
\quad 
q_j =
\begin{pmatrix}
a_j\\
1\\
b_j
\end{pmatrix},
\quad 
1\le j \le n.
$$
Thus, as suggested by the statement of \Cref{problem:flatland}, three flatland cameras are needed in order to make the flatlander reconstruction problem interesting.

The main results of this paper imply, as a special case, an \emph{implicit characterization} of the solutions to \Cref{problem:flatland}.
Following~\cite{agarwal2024atlas}, we study the \emph{flatlander image variety}: letting $\mathbb{P}^{2\times 3-1}\cong \mathbb{P}^5$ denote the projective space of all $2\times 3$ matrices, this is the closure of the image of the rational map
\begin{align}
\Theta_{n} : \left( \P^2 \right)^{\times n} \times \left( \P^{2\times 3-1}\right)^{\times 3} &\dashrightarrow \left( \P^{1}\right)^{\times 3n} \nonumber  \\
(q_1, \ldots , q_n, A_1, A_2, A_3 ) &\mapsto \left(A_i q_j \mid 1\le i \le 3 , 1\le j \le n \right)  \label{eq:flatland-image-map}.
\end{align}

Here, and throughout the paper, we work with complex algebraic varieties, so that the closure $\overline{\operatorname{im} \left(\Theta_n\right)}$ may be understood in either the Zariski or the analytic topology.
Our main results imply, as a special case, that the flatlander image variety is a determinantal variety with especially nice equations.
More precisely, our results imply for all $n\ge 1$ that
\begin{equation}\label{eq:flatland-implicit}
\overline{\operatorname{im} \left(\Theta_n\right)}
=
\left\{ 
(p_{11}, \ldots , p_{3n}) \in \left(\mathbb{P}^1 \right)^{\times 3n}
\, \, \Bigg\vert
\, \, 
\operatorname{rank} \left( 
\begin{array}{ccc}
\sigma (p_{11}, p_{21}, p_{31})^T\\
\hline 
\vdots \\
\hline 
\sigma (p_{1n}, p_{2n}, p_{3n})^T
\end{array}
\right) < 8
\right\},
\end{equation}
where 
\begin{equation}\label{flatland:segre}
\sigma (p_{1j}, p_{2j}, p_{3j}) = 
p_{1j} \otimes p_{2j} \otimes p_{3j} 
\in \mathbb{P}^7
\end{equation}
denotes a $8\times 1$ vector representing the Segre embedding $\P^1 \times \P^1 \times \P^1 \hookrightarrow \P^7$.

Since the flatlander image variety in~\Cref{eq:flatland-implicit} is defined by the rank-deficiency of an $n \times 8$ matrix, we see in particular that \Cref{problem:flatland} can be solved in the affirmative for almost all choices of $n<8$ point-triples.
For $n=8,$ the flatlander image variety is a hypersurface in $\left( \mathbb{P}^1 \right)^{\times 24}.$
Its equation is given by the \emph{Segre determinant} studied extensively in recent work of Pratt~\cite{pratt2025segre}.

While proving~\Cref{eq:flatland-implicit}, it became natural for us to study not only the left side of this equation, which served as our original motivation, but also a more general class of \emph{Segre determinantal loci} suggested by the right side.
In \Cref{sec:sdl}, we systematically study the maximal minor ideals of Segre-determinantal matrices, proving that they are prime and Cohen-Macaulay of the expected height (\Cref{thm:sdl-irreducibility,thm:sdl-prime-cm}).
\Cref{sec:ugb} continues this study from a combinatorial angle: we prove \Cref{thm:ugb}, which states that the maximal Segre-determinantal minors form a \emph{universal Gr\"{o}bner basis}, i.e.~a Gr\"{o}bner basis for all possible monomial orders.
Finally, in \Cref{sec:sdl-in-cv}, we return to the original motivation from computer vision, presenting our solution for the imaging map in~\Cref{eq:flatland-image-map} (\Cref{thm:flatland}), and discussing further connections for future study.

\section{Segre-Determinantal Loci} \label{sec:sdl}

For a finite sequence of integers $m_{\bullet} = (m_1, \ldots , m_r)$ with $m_i \ge 1$ for all $i,$ consider the product of projective spaces $\P^{m_\bullet} := \P^{m_1} \times \cdots \times \P^{m_r}$ under its Segre embedding,
\begin{align}
    \sigma_{m_\bullet } : \P^{m_\bullet} &\to \P^{M (m_\bullet)} \nonumber \\
    \left( [p_{i, 0} : \, \cdots \, : p_{i, m_i}] \mid 1\le i \le r \right) &\mapsto \left[ \left(p_{1,0} \cdots p_{r,0}\right) : \, \cdots \, : \left( p_{1,m_1} \cdots p_{r,m_r}\right)\right], \label{eq:segre-map}
\end{align}
where we define
\begin{equation}\label{eq:M-of-m-bullet}
    M (m_\bullet)  = (m_1 + 1)  \cdots  (m_r + 1) - 1.
\end{equation}
For convenience, we also introduce the notation
\begin{equation}\label{eq:d-of-m-bullet}
d (m_\bullet) = m_1 + \ldots + m_r,
\end{equation}
the dimension of the \emph{Segre variety}, whose points $\sigma_{m_\bullet} (P) \in \P^{M(m_\bullet)}$ are naturally identified with their preimages  $P = (p_1, \ldots , p_r) \in \P^{m_\bullet}$ under $\sigma_{m_\bullet}.$
\begin{definition}\label{def:sdl}
For any $m_\bullet $ and $n \ge 1,$ we define
the \emph{Segre-determinantal locus} to be the smallest closed subvariety $\Seg_{m_{\bullet} , n}\subset (\P^{m_\bullet})^{\times n} $ containing all $p_\bullet = (P_1, \ldots , P_n) \in (\P^{m_\bullet})^{\times n} $
whose Segre images
$\sigma_{m_\bullet } (P_1), \ldots , \sigma_{m_\bullet } (P_n) \in \P^{M (m_\bullet)}$ lie on some common hyperplane.
\end{definition}

The name ``Segre-determinantal locus'' is derived from the case of a hypersurface in a product of $r=2$ projective spaces,
\begin{equation}\label{eq:segre-determinant}
\Seg_{(m_1, m_2) , n}\subseteq \left( \P^{m_1} \times \P^{m_2} \right)^{\times n},
\quad
n = M(m_\bullet) + 1 = (m_1+1) (m_2+1),
\end{equation}
whose equation is the \emph{Segre determinant} recently studied by Pratt~\cite{pratt2025segre}.
For $r>2$ factors, \Cref{thm:sdl-prime-cm} establishes that the equations of $\Seg_{m_\bullet , n}$ are themselves just particular specializations of some Segre determinant in more variables.
Besides their connection to algebraic vision, discussed further in \Cref{sec:sdl-in-cv}, we refer to~\cite{pratt2025segre} for a discussion of how Segre determinants connect to geometric invariant theory and the study of Chow-Lam forms, which arise as irreducible factors of certain specializations of Segre determinants.

\begin{remark}\label{rem:higher-codimension}
For $1\le k < M(m_\bullet ),$ one could of course study an even more general class of Segre-determinantal loci $\Seg_{m_{\bullet} , n}^k$, where we instead require that the Segre images of \Cref{def:sdl} lie on a common subspace of codimension $k$ (hence $\Seg_{m_{\bullet} , n}=\Seg_{m_{\bullet} , n}^{1}$).
However, besides the further notational complications this introduces, the results we obtain in this work do not easily generalize to this setting. See \Cref{ex:p111-four-minors-not-prime} for further discussion.
\end{remark}

Note that for $n\le M (m_\bullet)$ in \Cref{def:sdl}, we have $\Seg_{m_{\bullet}, n}\ = (\P^{m_\bullet})^{\times n}.$
Thus, the results that follow are mainly interesting when $n > M (m_\bullet) .$
We now state the first of these results.

\begin{theorem}
  \label{thm:sdl-irreducibility}
For any $m_\bullet = (m_1, \ldots ,m_r)$ and any choice of $n > M (m_\bullet)$, the Segre-determinantal locus $\Seg_{m_\bullet , n} \subset (\P^{m_\bullet})^{\times n}$ is an irreducible subvariety of codimension $n - M (m_\bullet).$
\end{theorem}
\begin{proof}
For $r=1,$ the equations of $\Seg_{m_\bullet , n}$ are those of the classical determinantal varieties, from which the result readily follows.
Thus, we assume $r>1$ for the remainder of the proof.

  Consider the incidence variety
  \[
    W:=\left\{((P_1,\ldots,P_n),H)\in (\P^{m_\bullet})^{\times n}\times \left(\P^{M(m_\bullet)}\right)^*
      \;\middle|\;
      \sigma_{m_\bullet}(P_i)\in H \text{ for all } i
    \right\}.
  \]
  Here \(\left(\P^{M(m_\bullet)}\right)^*\) denotes the dual projective space parametrizing hyperplanes in \(\P^{M(m_\bullet)}\). Let
  \begin{align*}
    \pi_1&:W\to (\P^{m_\bullet})^{\times n},
    &
      ((P_1,\ldots,P_n),H)&\mapsto (P_1,\ldots,P_n), \\
    \pi_2&:W\to \left(\P^{M(m_\bullet)}\right)^*,
    &
      ((P_1,\ldots,P_n),H)&\mapsto H
  \end{align*}
  be the two coordinate projections of \(W\).

  Since \(W\) is closed in the projective variety \(\left(\P^{m_\bullet}\right)^{\times n}\times \left(\P^{M(m_\bullet)}\right)^*\), the image \(\pi_1(W)\) is closed. By definition, \(\pi_1(W)\) is exactly the Segre-determinantal locus \(\Seg_{m_\bullet,n}\). Thus, to complete the proof, it will suffice to show that \(W\) is irreducible and to compute \(\dim \left(\pi_1(W)\right)\).

  Let \(X:=\Image(\sigma_{m_\bullet})\subset \P^{M(m_\bullet)}\) be the Segre variety. Since $X$ is a nondegenerate, irreducible, and smooth projective variety of dimension \(d (m_\bullet) \), a generic hyperplane section $X\cap H$ has pure dimension $d(m_\bullet ) - 1$. More precisely, by Bertini's theorem \cite[Théorème~6.3(4)]{jouanolou1983}, there exists a dense open subset $\mathcal{U} \subset \left(\P^{M(m_\bullet)}\right)^*$ such that for all $H\in \mathcal{U}$, the hyperplane section $X\cap H$ is irreducible.
  For such a hyperplane $H \in \mathcal{U}$, the fiber of $\pi_2$ is
  \[
    \pi_2^{-1}(H)
    =
    \{(P_1,\ldots,P_n)\in (\P^{m_\bullet})^{\times n} : \sigma_{m_\bullet}(P_i)\in H
    \text{ for all }i\}.
  \]
  The Segre embedding $\sigma_{m_\bullet }$ is a (biregular) isomorphism of varieties, so applying \(\sigma_{m_\bullet}\) in each factor gives an isomorphism
  \begin{equation}
    \label{eq:fiberIsoProdIntersections}
    \pi_2^{-1}(H) \xrightarrow{\sim} (X\cap H)^{\times n},\qquad
    (P_1,\ldots,P_n)\longmapsto
    \bigl(\sigma_{m_\bullet}(P_1),\ldots,\sigma_{m_\bullet}(P_n)\bigr).
  \end{equation}
  Since a finite product of irreducible varieties is irreducible (see \cite[\S 3.1, Theorem 3]{Shafarevich94}), it follows that $\pi_2^{-1}(H)$ is irreducible of dimension \(n(d(m_\bullet)-1)\).
  Since $r>1,$ we have $d(m_\bullet )\ge 2,$ and hence 
  $$
\dim \left( \pi_2^{-1}(H) \right)  = n(d(m_\bullet)-1) > 1.
  $$
  Consider the open subset $W^{\mathrm{O}} := \pi_2^{-1}(\mathcal{U}) \subset W$.
  Since $\mathcal{U}$ is irreducible and $\pi_2$ is a morphism with irreducible fibers of constant dimension over $\mathcal{U}$, it follows from the dimension estimate above and the irreducibility criterion for fibered varieties \cite[Ch.~I, \S6, Theorem 7 (ii)]{Shafarevich94} that $W^{\mathrm{O}}$ is irreducible of dimension
  \begin{align*}
    \dim(W^{\mathrm{O}}) &= \dim(\mathcal{U})+n(d(m_\bullet)-1) \\
                         &=  M(m_\bullet)+n(d(m_\bullet)-1).
  \end{align*}
  To show that $W$ is irreducible, let \(W^{\prime}\) be an irreducible component of \(W\). The ambient variety \((\P^{m_\bullet})^{\times n}\times (\P^{M(m_\bullet)})^*\) has dimension \(nd(m_\bullet)+M(m_\bullet)\), and \(W\) is a subvariety satisfying \(n\) Zariski-closed conditions \(\sigma_{m_\bullet}(P_i)\in H\) (for \(1\le i\le n\)). 
  Hence every irreducible component of \(W\) has codimension at most \(n\) (see e.g.~\cite[Ch.~I, \S6, Corollary 5]{Shafarevich94}), and so
  \begin{equation}\label{eq:dimW'-lower}
    \dim(W^{\prime}) \ge nd(m_\bullet )+M(m_\bullet)-n = M(m_\bullet) +n(d(m_\bullet)-1).
  \end{equation}
  On the other hand, since \(X\) is nondegenerate, \(\dim(X \cap H) \leq d(m_\bullet) - 1\) for any hyperplane \(H\).  By~\Cref{eq:fiberIsoProdIntersections}, \(\pi_{2}^{-1}(H) \cong (X \cap H)^{\times n}\) and so
  \begin{equation}\label{eq:dim-fiberpi2-upper}\dim(\pi_{2}^{-1}(H)) \leq n(d(m_\bullet)-1).\end{equation}
  Combining this last inequality with the fiber-dimension theorem (e.g.~\cite[Ch.~I, \S6, Corollary 5]{Shafarevich94}) gives us the bound
  \begin{equation}\label{eq:dimW'-upper}\dim(W^{\prime}) \le \dim\left(\overline{\pi_{2}(W^{\prime})}\right) + n(d(m_\bullet)-1) \le M(m_\bullet)+n(d(m_\bullet) - 1).
  \end{equation}
  Combining the inequalities in~\Cref{eq:dimW'-lower} and~\Cref{eq:dimW'-upper}, we obtain
  \begin{align*}
    \dim(W^{\prime}) &= M(m_\bullet) +n(d(m_\bullet)-1),\\
    \dim\left(\overline{\pi_{2}(W^{\prime})}\right) &= M(m_\bullet).
  \end{align*}
  Thus \(\pi_2(W')\) is dense in \(\left(\P^{M(m_\bullet)}\right)^*\), so \(W'\) meets \(W^{\mathrm{O}}\). Since \(W^{\mathrm{O}}\) is open in \(W\), the intersection \(W' \cap W^{\mathrm{O}}\) is a nonempty open subset of the irreducible variety \(W'\), and is therefore dense in \(W'\). Since $W'$ was an arbitrary irreducible component of $W,$ we have thus shown that every irreducible component of \(W\) is contained in \(\overline{W^{\mathrm{O}}}\), and hence \(W = \overline{W^{\mathrm{O}}}\) is irreducible.

  Finally, it remains to compute the (co)dimension of \(\Seg_{m_\bullet,n}\). Note that the fiber of \(\pi_1\) over a point \((P_1,\ldots,P_n)\in\Seg_{m_\bullet,n}\) is
  \[
    \pi_1^{-1}(P_1,\ldots,P_n)
    =
    \bigl\{H\in \left(\P^{M(m_\bullet)}\right)^*:\sigma_{m_\bullet}(P_i)\in H\text{ for all }i\bigr\},
  \]
  i.e.~the set of all hyperplanes containing the linear space 
  $$
  \langle\sigma_{m_\bullet}(P_1),\ldots,\sigma_{m_\bullet}(P_n)\rangle.
  $$
  Let us choose points \(P_1,\ldots,P_{M(m_\bullet)}\in \mathbb{P}^{m_\bullet}\) so that their Segre images are in linear general position, and consider the hyperplane that they span,
  $$
H_0 = \langle\sigma_{m_\bullet}(P_1),\ldots,\sigma_{m_\bullet}(P_{M (m_\bullet )})\rangle .
  $$
Choosing additional points
  $$
  P_{M(m_\bullet)+1},\ldots,P_n \in \sigma_{m_\bullet}^{-1}(X\cap H_0),
  $$
  we obtain a point on the Segre-determinantal locus $$ (P_1, \ldots , P_n ) \in  \Seg_{m_\bullet,n},$$ whose fiber under \(\pi_1\) is the singleton \(\{H_0\}\). 
  Since $\pi_1$ surjects onto $\Seg_{m_\bullet , n}$, the fiber-dimension theorem implies the desired conclusion that
  \[
    \dim(\Seg_{m_\bullet,n})
    =
    \dim(W)
    =
    M(m_\bullet)+n(d(m_\bullet)-1).
  \]
\end{proof}

To make the preceding proof more concrete, we consider an example.

\begin{example}\label{ex:sdl-flatland-special-case}
Consider the sequence $m_\bullet=(1,1,1)$, so that
\[
  \P^{m_\bullet}=\P^1\times\P^1\times\P^1,
  \qquad
  d(m_\bullet)=1+1+1=3,
  \qquad
  M(m_\bullet)=(1+1)^3-1=7,
\]
and \(\sigma_{m_\bullet }  = \sigma_{(1,1,1)} \) is the Segre embedding
\[
  \P^1\times\P^1\times\P^1 \hookrightarrow \P^7.
\]
The Segre-determinantal locus \(\Seg_{(1,1,1),\, n}\subset \left(\P^1\times\P^1\times\P^1\right)^{\times n}\) parametrizes \(n\)-tuples \(P_1,\ldots,P_n\) whose Segre images \(\sigma_{(1,1,1)}(P_1),\ldots,\sigma_{(1,1,1)}(P_n)\) lie on a common hyperplane in \(\P^7\).
As we will see later in \Cref{thm:flatland}, this is precisely the Zariski closure of the image of the flatlander imaging map in~\Cref{eq:flatland-image-map} from the introduction.

The incidence variety appearing in the proof of \Cref{thm:sdl-irreducibility} is
\[
  \left\{
    ((P_1,\ldots,P_n),H)\in
    (\P^1\times\P^1\times\P^1)^{\times n}\times(\P^7)^*
    \;\middle|\;
    \sigma_{(1,1,1)}(P_i)\in H \text{ for all } i
  \right\}.
\]
For a generic hyperplane \(H\), Bertini's theorem implies the intersection
\[
  \Image(\sigma_{(1,1,1)})\cap H
\]
is an irreducible surface, and hence $\pi_2^{-1} (H)$ is irreducible of dimension \(2n\).
Therefore, the proof of \Cref{thm:sdl-irreducibility} gives the dimension count
\[
  \dim \Seg_{(1,1,1),n}=2n+7, 
  \quad n\ge 7.
\]
In particular, \(\Seg_{(1,1,1),8} \subset (\P^1\times\P^1\times\P^1)^{\times 8}\) is a hypersurface.
\end{example}

Having shown irreducibility of the Segre-determinantal loci in \Cref{thm:sdl-irreducibility}, we now concentrate on describing their prime vanishing ideals with our next result, \Cref{thm:sdl-prime-cm}.

Retaining our previous notation, consider the polynomial ring with complex coefficients in \( n(d(m_\bullet)+r) \) variables defined as 
\begin{equation}
    \label{eq:poly-ring}
    R_{m_\bullet , n} = \CC \left[ p_{j,i}^{(k)} \mid  1\le k\le n, 1\le j\le r, 0\le i\le m_j \right].
\end{equation} 

 Let \(\mathcal{N}_{m_\bullet,n}(p_\bullet )\) be the matrix whose rows are the vectors \(\sigma_{m_\bullet}(P_k)\), for $k=1, \ldots , n$, in the standard tensor-product basis, i.e.
\begin{equation}\label{eq:Nmn}
\mathcal{N}_{m_\bullet,n}(p_\bullet ) =
\begin{pmatrix}
    \sigma_{m_\bullet } (P_1)^T \\
    \hline 
    \vdots \\
    \hline 
    \sigma_{m_\bullet } (P_n)^T
\end{pmatrix} \in 
\left( 
R_{m_\bullet , n}
\right)^{n \times \bigl(M(m_\bullet)+1\bigr) },
\end{equation}
and let $I_{m_\bullet , n} \subset R_{m_\bullet , n}$ denote the ideal of $\left( M(m_\bullet ) + 1\right) \times \left( M(m_\bullet ) + 1\right)$ minors of this matrix.

\begin{theorem}
  \label{thm:sdl-prime-cm}
  For $n \ge M(m_\bullet ) + 1,$ the maximal minor ideal $I_{m_\bullet , n}$ is prime and of the expected height $n-M(m_\bullet ),$ and the quotient ring $R_{m_\bullet , n}/I_{m_\bullet , n}$ is Cohen-Macaulay.
  Thus, by \Cref{thm:sdl-irreducibility}, the maximal minors of $\mathcal{N}_{m_\bullet,n}(p_\bullet )$ generate the vanishing ideal of $\Seg_{m_\bullet , n}.$
\end{theorem}
\begin{proof}
The multiprojective vanishing locus $\mathcal{V}(I_{m_\bullet , n})$ and $\Seg_{m_\bullet , n}$ are manifestly equal; in the notation of the proof of \Cref{thm:sdl-irreducibility}, both varieties are the image of the incidence variety $W$ under $\pi_1.$
Using an incidence-variety argument entirely analogous to that used in the last proof, we may see that the affine variety $\mathcal{V}_{\text{aff}} (I_{m_\bullet , n}) \subset \mathbb{C}^{n(d(m_\bullet)+r)}$ is also irreducible.

To compute the height of the ideal $I_{m_\bullet , n}$, we subtract the dimension of the affine cone over $\Seg_{m_\bullet , n}$ from the total number of variables in $R_{m_\bullet , n}$: this gives 
\begin{align}\label{eq:height}
    \height \left( I_{m_\bullet , n} \right) &=
    n(d(m_\bullet)+r) - 
    \left( 
M(m_\bullet ) + n (d(m_\bullet )-1) + nr
    \right)\\
 &=n-M(m_\bullet). \nonumber 
  \end{align}
  Since $n \ge M(m_\bullet ) + 1,$ this is the expected height for the ideal of maximal minors for a matrix of polynomials with size \(n\times(M(m_\bullet)+1)\). By the Eagon-Northcott perfection theorem (see e.g.~\cite[Theorem~2.7]{bruns88_deter_rings}), the ideal \(I_{m_\bullet , n}\) is perfect. Since \(R_{m_\bullet , n}\) is a Cohen-Macaulay ring, it follows that the quotient ring \(R_{m_\bullet , n}/I_{m_\bullet , n}\) is also Cohen-Macaulay.

It remains to show that \(I_{m_\bullet , n} \) is prime. 
 By the Unmixedness Theorem (\cite[Theorem 18.14]{eisenbud2013commutative}), all associated primes of $I_{m_\bullet , n}$ have height $n - M(m_\bullet ).$ 
 Since $\mathcal{V}_{\text{aff}} (I_{m_\bullet , n})$ is irreducible, it will suffice by the Nullstellensatz to show that $I_{m_\bullet , n}$ is radical.
We will prove this via the Jacobian criterion for reducedness of Cohen-Macaulay rings~\cite[Theorem~18.15(a)]{eisenbud2013commutative}.  We index the columns of \(\mathcal{N}_{m_\bullet,n}(p_\bullet )\) by exponent vectors 
$$
\alpha=(\alpha^{(1)},\ldots,\alpha^{(r)})\in \mathbb{Z}_{\ge 0}^{r} , \quad 0\le \alpha^{(j)}\le m_j.
$$
We associate to any such exponent vector the point
$$
P_{\alpha }
=
(e_{\alpha^{(1)}}, \ldots , e_{\alpha^{(r)}}) \in \mathbb{P}^{m_\bullet },
$$
where each $e_{\alpha^{(j)}}$ is a projective standard basis vector of the appropriate size and in the usual ordering.
Let us enumerate the points $P_\alpha $ in the lexicographic order, 
\begin{align*}
    P_{\alpha_0} &= (e_0, \ldots , e_0 ),\\  
    P_{\alpha_1} &= (e_0, \ldots , e_1),\\
    &\, \, \,  \vdots \\
    P_{\alpha_{M(m_\bullet )}} &= (e_{m_1}, \ldots , e_{m_r}), 
\end{align*}
and set 
$$
\beta = (1,0,\ldots , 0) \in \mathbb{Z}_{\ge 0}^r
\quad 
\Rightarrow 
\quad 
P_\beta = (e_1, e_0, \ldots, e_0) \in \mathbb{P}^{m_\bullet } .
$$
We now construct a point $q_\bullet \in \Seg_{m_\bullet , n}$ by collecting all but the first of these points in order and then repeating $P_{\beta}$ in the last $n - M(m_\bullet )$ factors:
$$
q_\bullet = \left( 
P_{\alpha_1},
\ldots , 
P_{\alpha_{M(m_\bullet )}},
P_{\beta},
\ldots  
P_{\beta}
\right) \in \Seg_{m_\bullet , n} .
$$
For \(k=M(m_\bullet)+1,\ldots,n\), let \(f_k\) be the maximal minor of $\mathcal{N}_{m_\bullet , n}(p_\bullet)$ which uses the first \(M(m_\bullet)\) rows together with row \(k\), and consider the Jacobian submatrix at $q_\bullet$ given by
$$
J_{q_\bullet} = \left( \displaystyle\frac{\partial f_k}{\partial p_{1,0}^{(\ell )}} (q_\bullet )\right)_{\substack{M(m_\bullet) < \ell , k \le n\\}} \in \CC^{(n - M(m_\bullet) ) \times (n - M(m_\bullet ))}.
$$
It is easy to see that $J_{q_\bullet}$ is a diagonal matrix, since the $\ell $-th row of this matrix does not involve the variables $p_{1,0}^{(\ell )}$ for $\ell > M(m_\bullet)$ with $\ell \ne k .$ 
To complete the proof, it will be enough to show that the diagonal entries of $J_{q_\bullet}$ are all $\pm 1.$ 
This can be seen by computing $f_k$ by cofactor expansion along the $k$-th row: for an appropriate choice of signs, we have
$$
f_k (p_\bullet) = 
\displaystyle\sum_{\alpha } 
\operatorname{sgn} (\alpha ) \cdot  \textrm{cof}_{\alpha , k} (p_\bullet ) \cdot  p_{1,\alpha^{(1)}}^{(k)} \cdots p_{r,\alpha^{(r)}}^{(k)}.
$$
Note that each cofactor $\textrm{cof}_{\alpha , k} (p_\bullet)$ does not involve any variables of the form $p_{j,i}^{(k)},$ i.e.~those variables appearing in the $k$-th row of $\mathcal{N}_{m_\bullet,n}(p_\bullet)$.
Thus, we may differentiate $f_k$ to obtain
\begin{equation}\label{eq:partial-derivative}
\displaystyle\frac{\partial f_k}{\partial p_{1,0}^{(k)}}
=
\displaystyle\sum_{\alpha \, : \, \alpha^{(1)} = 0} 
\operatorname{sgn} (\alpha ) \, \textrm{cof}_{\alpha , k} (p_\bullet)\, p_{2,\alpha^{(2)}}^{(k)} \cdots p_{r,\alpha^{(r)}}^{(k)}.
\end{equation}
When we evaluate the partial derivative in~\Cref{eq:partial-derivative} at $q_\bullet,$ all summands with $\alpha^{(j)} >0 $ for some $2\le j \le r$ will vanish, due to the zero-structure of $P_\beta .$
Hence 
\begin{equation}\label{eq:partial-derivative-eval}
\displaystyle\frac{\partial f_k}{\partial p_{1,0}^{(k)}} (q_\bullet)
=
\pm \, \textrm{cof}_{\alpha_0, k} (q_\bullet).
\end{equation}
Finally, it remains to observe that $\textrm{cof}_{\alpha_0, k} (q_\bullet) = \pm 1$.
To see this, simply note that by the construction of $P_{\alpha_1}, \ldots , P_{\alpha_{M(m_\bullet )}},$ the submatrix of $\mathcal{N}_{m_\bullet , n}(q_\bullet)$ indexed by columns $\alpha_1, \ldots , \alpha_{M(m_\bullet )}$ and rows $1, \ldots , M(m_\bullet )$ is simply the $M(m_\bullet )\times M(m_\bullet)$ identity matrix.
\end{proof}

\begin{example}\label{ex:p111-jacobian-point}
  We illustrate the Jacobian calculation above in the case \(m_\bullet=(1,1,1)\) and
  \(n=9\).  Here \(d(m_\bullet)=3\), \(r=3\), \(M(m_\bullet)=7\), \( \dim( R_{m_\bullet , n}) =9(3+3)=54\), and hence
  \[
    \height (I_{m_\bullet , n}) = n-M(m_\bullet)=9-7=2.
  \]
    We order the columns of the matrix $\mathcal{N}_{m_
    \bullet , n}(p_\bullet)$ lexicographically, 
  \begin{align*}
    \alpha_0 = (0,0,0), 
    \quad  
    \alpha_1 = (0,0,1),
    &\quad 
    \alpha_2 = (0,1,0), 
    \quad  
    \alpha_3 = (0,1,1),
    \\\ 
    \alpha_4 = (1,0,0), 
    \quad 
    \alpha_5 = (1,0,1),
    &\quad 
    \alpha_6 = (1,1,0), 
    \quad 
    \alpha_7 = (1,1,1),  
  \end{align*}
so that
\begin{equation}
\mathcal{N}_{m_
    \bullet , n} (p_\bullet)
    =
{    \scriptsize 
    \begin{pmatrix}
    p_{1,0}^{(1)}p_{2,0}^{(1)}p_{3,0}^{(1)}&
    p_{1,0}^{(1)}p_{2,0}^{(1)}p_{3,1}^{(1)}&
    p_{1,0}^{(1)}p_{2,1}^{(1)}p_{3,0}^{(1)}&
    p_{1,0}^{(1)}p_{2,1}^{(1)}p_{3,1}^{(1)}&
    p_{1,1}^{(1)}p_{2,0}^{(1)}p_{3,0}^{(1)}&
    p_{1,1}^{(1)}p_{2,0}^{(1)}p_{3,1}^{(1)}&
    p_{1,1}^{(1)}p_{2,1}^{(1)}p_{3,0}^{(1)}&
    p_{1,1}^{(1)}p_{2,1}^{(1)}p_{3,1}^{(1)}\\
    \vdots & \vdots & \vdots & \vdots &   \vdots & \vdots & \vdots & \vdots \\
        p_{1,0}^{(9)}p_{2,0}^{(9)}p_{3,0}^{(9)}&
    p_{1,0}^{(9)}p_{2,0}^{(9)}p_{3,1}^{(9)}&
    p_{1,0}^{(9)}p_{2,1}^{(9)}p_{3,0}^{(9)}&
    p_{1,0}^{(9)}p_{2,1}^{(9)}p_{3,1}^{(9)}&
    p_{1,1}^{(9)}p_{2,0}^{(9)}p_{3,0}^{(9)}&
    p_{1,1}^{(9)}p_{2,0}^{(9)}p_{3,1}^{(9)}&
    p_{1,1}^{(9)}p_{2,1}^{(9)}p_{3,0}^{(9)}&
    p_{1,1}^{(9)}p_{2,1}^{(9)}p_{3,1}^{(9)}
    \end{pmatrix}
    }.\label{eq:N-specific}
\end{equation}
In the notation of the proof, we have $\beta = \alpha_4 = (1,0,0),$ and $q_\bullet \in \Seg_{(1,1,1) , 9}$ is given as follows: 
\begin{align*}
P_1 = \left( [1:0], [1:0], [0:1] \right), &\quad 
P_2 = \left( [1:0], [0:1], [1:0] \right),
\quad 
P_3 = \left( [1:0], [0:1], [0:1] \right),\nonumber \\
P_5 = \left( [0:1], [1:0], [0:1] \right),
&\quad 
P_6 = \left( [0:1], [0:1], [1:0] \right),
\quad 
P_7 = \left( [0:1], [0:1], [0:1] \right), \nonumber \\
P_4 = P_8 = P_9 = \left( [0:1], [1:0], [1:0] \right),
&\quad 
q_\bullet = (P_1, P_2, P_3, P_4, P_5, P_6, P_7, P_4, P_4).
\end{align*}
Evaluating the matrix in~\Cref{eq:N-specific} at $q_\bullet$ gives
\begin{equation}
    \label{eq:N-specific-eval}
    \mathcal{N}_{m_
    \bullet , n} (q_\bullet) = 
    \begin{pmatrix}
        0 & 1 & 0 & 0 & 0 & 0 & 0 & 0 \\
        0 & 0 & 1 & 0 & 0 & 0 & 0 & 0 \\
        0 & 0 & 0 & 1 & 0 & 0 & 0 & 0 \\
        0 & 0 & 0 & 0 & 1 & 0 & 0 & 0 \\
        0 & 0 & 0 & 0 & 0 & 1 & 0 & 0 \\
        0 & 0 & 0 & 0 & 0 & 0 & 1 & 0 \\
        0 & 0 & 0 & 0 & 0 & 0 & 0 & 1 \\
        0 & 0 & 0 & 0 & 1 & 0 & 0 & 0 \\
        0 & 0 & 0 & 0 & 1 & 0 & 0 & 0 
    \end{pmatrix}
\end{equation}
For the maximal minors $f_1, f_2 \in I_{(1,1,1), 9}$ defined by deleting rows $9$ and $8$ from~\Cref{eq:N-specific}, respectively,
a straightforward computation gives
\begin{equation}
J_{q_\bullet} = 
\begin{pmatrix}
    \displaystyle\frac{\partial f_1}{\partial p_{1,0}^{(8)}} (q_\bullet) &  
    \displaystyle\frac{\partial f_2}{\partial p_{1,0}^{(8)}} (q_\bullet)
    \\
        \displaystyle\frac{\partial f_1}{\partial p_{1,0}^{(9)}} (q_\bullet) &  
    \displaystyle\frac{\partial f_2}{\partial p_{1,0}^{(9)}} (q_\bullet)
\end{pmatrix}
= \begin{pmatrix}
    -1 & 0 \\
    0 & -1 
\end{pmatrix},
\end{equation}
and thus $\operatorname{rank} (J_{q_\bullet}) = \height (I_{m_\bullet , n})=2.$
\end{example}

In view of \Cref{rem:higher-codimension}, it is natural to ask to what extent \Cref{thm:sdl-irreducibility,thm:sdl-prime-cm} can be generalized to non-maximal minor ideals for the Segre-determinantal matrix in~\Cref{eq:Nmn}.
We close this section with an example showing that such a generalization would not be completely straightforward.

\begin{example}[A non-prime ideal of \(4\times 4\) minors]
  \label{ex:p111-four-minors-not-prime}
  As before, we take \(m_\bullet=(1,1,1)\), so \(M(m_\bullet)=7\).
For $n\ge 4,$ the $4\times 4$ minors of the $n\times 8$ matrix in~\Cref{eq:Nmn} do not generate a prime ideal.
To see this, we consider the \(4\times4\) minor \(\Delta\) that uses the first four rows and the first four columns of $\mathcal{N}_{m_\bullet , n}(p_\bullet)$,
  \[
    \Delta=
    \det\begin{pmatrix}
      p_{1,0}^{(1)}p_{2,0}^{(1)}p_{3,0}^{(1)} & p_{1,0}^{(1)}p_{2,0}^{(1)}p_{3,1}^{(1)} & p_{1,0}^{(1)}p_{2,1}^{(1)}p_{3,0}^{(1)} & p_{1,0}^{(1)}p_{2,1}^{(1)}p_{3,1}^{(1)} \\
      p_{1,0}^{(2)}p_{2,0}^{(2)}p_{3,0}^{(2)} & p_{1,0}^{(2)}p_{2,0}^{(2)}p_{3,1}^{(2)} & p_{1,0}^{(2)}p_{2,1}^{(2)}p_{3,0}^{(2)} & p_{1,0}^{(2)}p_{2,1}^{(2)}p_{3,1}^{(2)} \\
      p_{1,0}^{(3)}p_{2,0}^{(3)}p_{3,0}^{(3)} & p_{1,0}^{(3)}p_{2,0}^{(3)}p_{3,1}^{(3)} & p_{1,0}^{(3)}p_{2,1}^{(3)}p_{3,0}^{(3)} & p_{1,0}^{(3)}p_{2,1}^{(3)}p_{3,1}^{(3)} \\
      p_{1,0}^{(4)}p_{2,0}^{(4)}p_{3,0}^{(4)} & p_{1,0}^{(4)}p_{2,0}^{(4)}p_{3,1}^{(4)} & p_{1,0}^{(4)}p_{2,1}^{(4)}p_{3,0}^{(4)} & p_{1,0}^{(4)}p_{2,1}^{(4)}p_{3,1}^{(4)}
    \end{pmatrix}.
  \]
Using multilinearity of the determinant in rows, this factors as 
  \[
\Delta = p_{1,0}^{(1)}\cdot p_{1,0}^{(2)}\cdot p_{1,0}^{(3)}\cdot p_{1,0}^{(4)} 
\cdot 
    \det\begin{pmatrix}
p_{2,0}^{(1)}p_{3,0}^{(1)} & p_{2,0}^{(1)}p_{3,1}^{(1)} & p_{2,1}^{(1)}p_{3,0}^{(1)} & p_{2,1}^{(1)}p_{3,1}^{(1)} \\
p_{2,0}^{(2)}p_{3,0}^{(2)} & p_{2,0}^{(2)}p_{3,1}^{(2)} & p_{2,1}^{(2)}p_{3,0}^{(2)} & p_{2,1}^{(2)}p_{3,1}^{(2)} \\
p_{2,0}^{(3)}p_{3,0}^{(3)} & p_{2,0}^{(3)}p_{3,1}^{(3)} & p_{2,1}^{(3)}p_{3,0}^{(3)} & p_{2,1}^{(3)}p_{3,1}^{(3)} \\
p_{2,0}^{(4)}p_{3,0}^{(4)} & p_{2,0}^{(4)}p_{3,1}^{(4)} & p_{2,1}^{(4)}p_{3,0}^{(4)} &
p_{2,1}^{(4)}p_{3,1}^{(4)}
    \end{pmatrix}.
  \]
   By degree considerations, none of the factors belong to the $4\times 4$ minor ideal. Hence this ideal is not prime.
\end{example}

\section{Universal Gr\"obner Bases of Segre-Determinantal Loci} \label{sec:ugb}

In this section, we prove that the maximal Segre-determinantal minors form a universal Gr\"{o}bner basis. 
Recall that a \emph{universal Gr\"{o}bner basis} for an ideal $I\subset \mathbb{C}[x_1,\ldots , x_n]$ is a Gr\"{o}bner basis for all possible monomial orders.
Universal Gr\"{o}bner bases always exist, but are generally quite difficult to compute---nevertheless, they have been determined for some especially nice determinantal ideals, such as \Cref{ex:2x2-nonexample,ex:maximal-example} below.
Our result utilizes a combinatorial criterion for universal Gr\"{o}bner bases developed by Huang and Larson, stated as \Cref{thm:larson-huang} below, whose statement requires some rudiments of simplicial complexes and matroid theory.

An \emph{abstract simplicial complex} $\Delta $ over a finite ground set $S$ is a downward-closed family of subsets of $S$; that is, $\Delta \subset 2^S$ such that $A \in \Delta $ whenever $A\subset B$ and $B\in \Delta .$
Usually we simply refer to $\Delta $ as a \emph{complex}.
The elements of $\Delta $ are \emph{faces}, and inclusion-maximal faces are \emph{facets.}

The complexes that interest us will mainly be \emph{matroidal independence complexes}.
This means complexes $\Delta $ which satisfy the \emph{exchange axiom}, which states that if $A ,B \in \Delta $ with $ |A| < |B|,$ then there exists an $a \in B \setminus A$ with $A \cup \{ a \} \in \Delta .$
It is a standard fact of matroid theory that independence complexes are \emph{pure}, meaning that all facets have the same cardinality.
The facets of an independence complex are known as \emph{bases}.
A \emph{matroid} may be defined in several equivalent ways: for instance, via its independence complex, or via its bases, or via its \emph{circuits}, which may be understood as the minimal non-faces of a matroidal independence complex.

We now introduce two complexes that play main roles in our results.  The first complex is defined via algebraic data. If $f_1, \ldots , f_k \in \mathbb{C} [x_1, \ldots , x_n]$ are nonzero polynomials whose terms are all squarefree monomials, we say that $F = \{ f_1, \ldots , f_k \}$ is a \emph{squarefree set of polynomials.} 

\begin{definition}\label{def:spread-complex}
The \emph{spread complex} $\Delta (F)$ associated to a squarefree set of polynomials $F = \{ f_1, \ldots , f_k \}$ is the complex on the ground set of variables $S = \{ x_1, \ldots , x_n \}$ whose minimal non-faces are the \emph{spreads} (aka supports),
\begin{equation}
    \label{eq:spread}
    \spr (f_i ) = \{ x_j \text{ appears in some monomial of } f_i \} , \quad 1\le i \le k.
\end{equation}
\end{definition}

\begin{example}
  \label{ex:flatland-spreads}
  Take \(m_\bullet=(1,1,1)\).  Then \(M(m_\bullet)=7\), so the maximal minors of the matrix \(\mathcal{N}_{(1,1,1),n} (p_\bullet) \) are \(8\times 8\). The variables in the \(k\)th row are
  \[
    p_{1,0}^{(k)},\ p_{1,1}^{(k)},\
    p_{2,0}^{(k)},\ p_{2,1}^{(k)},\
    p_{3,0}^{(k)},\ p_{3,1}^{(k)}.
  \]
  If \(K\subset [n]\) is a set of eight rows, then the corresponding maximal minor uses only the rows indexed by \(K\).  Its spread is exactly
  \[
    \left\{
      p_{j,i}^{(k)}
      \;\middle|\;
      k\in K,\ 1\le j\le 3,\ 0\le i\le 1
    \right\},
  \]
  the set of all six variables from each of the eight selected rows.

  For instance, when \(n=9\), consider the set \(U\) containing all variables from rows \(1,\ldots,7\), all variables from row \(8\) except \(p_{1,0}^{(8)}\), and all variables from row \(9\) except \(p_{1,0}^{(9)}\).  This set contains no spread of a maximal minor because a spread requires eight complete rows, while \(U\) has only seven complete rows. However, adding either omitted variable completes one of rows \(8\) or \(9\), and then \(U\) contains the spread of an \(8\times 8\) minor.  Thus \(U\) is a facet of the spread complex.  Its size is
  \[
    7\cdot 6+2\cdot 5=52
    =
    M(m_\bullet)+n\bigl(d(m_\bullet)+r-1\bigr)
    =
    7+9\cdot 5.
  \]
\end{example}

The second complex that interests us is a much more geometric object---the algebraic matroid of an irreducible embedded affine variety. 
If we are given such a variety $X \subset \mathbb{C}^n$, then the field of rational functions on $X$ is an image of the $n$-variable rational function field under a restriction map
\begin{equation}\label{eq:restriction-to-X}
\operatorname{res}_X : \mathbb{C} (x_1, \ldots , x_n) \to  \mathbb{C} (X ).
\end{equation}

\begin{definition}\label{def:alg-matroid}
Given an irreducible subvariety of affine space $X \subset \mathbb{C}^n$, we define the \emph{algebraic matroid} $\Delta (X )$ to be the matroidal independence complex on the ground set of variables $\{ x_1, \ldots , x_n \} $ whose faces are the subsets of variables whose images under $\operatorname{res}_X$ are algebraically independent.
\end{definition}

The fact that $\Delta (X)$ is a matroidal independence complex is a standard algebraic fact. 
More geometrically, the faces of $\Delta (X)$ may be characterized in terms of projections as follows: for any subset $U \subset [n] := \{ 1, \ldots , n\}$, we let 
\begin{align}
    \pi_U : X &\to \mathbb{C}^{|U|} \nonumber \\
    (x_1, \ldots , x_n ) &\mapsto (x_i \mid i \in U)
\end{align}
denote projection onto the coordinates indexed by $U.$ We then have that $U \in \Delta (X)$ if and only if $\pi_U$ is dominant.

The following result provides sufficient conditions for a squarefree set of polynomials $F$ vanishing on irreducible $X \subset \mathbb{C}^n$ to form a \emph{universal Gr\"{o}bner basis} for the vanishing ideal $\mathcal{I} (X)$---that is, a Gr\"{o}bner basis for all possible term orders on the polynomial ring $\mathbb{C} [x_1, \ldots , x_n].$ 

\begin{theorem}\label{thm:larson-huang}\cite{huang2024fine}
Let $X \subset \mathbb{C}^n$ be an irreducible closed subvariety of affine space, and let $F \subset \mathcal{I}(X)$ be a squarefree set of polynomials vanishing on $X.$ If the algebraic matroid of $X$ and the spread complex are equal, i.e. if~$\Delta (X) = \Delta (F) $, then $F$ forms a universal Gr\"{o}bner basis of $\mathcal{I}(X).$ 
\end{theorem}

\Cref{thm:larson-huang}, as stated above, is a special case of a more general universal Gr\"{o}bner basis criterion due to Huang and Larson~\cite{huang2024fine}, which allows for more general varieties over arbitrary fields.
The special case stated above is sufficient for our purposes; a somewhat simpler proof in this case is given in ~\cite{duff2025universal}.

The main difficulty in applying \Cref{thm:larson-huang} is proving that the spread complex is contained in the algebraic matroid,
\begin{equation}\label{eq:hard-inclusion}
\Delta (F ) \subset \Delta (X).
\end{equation}
since the reverse containment
\begin{equation}\label{eq:easy-inclusion}
\Delta (X ) \subset \Delta (F)
\end{equation}
always holds when the other hypotheses of \Cref{thm:larson-huang} are satisfied.
To see the inclusion in~\Cref{eq:easy-inclusion}, it is equivalent to observe that any non-face of the spread complex must also be a non-face of the algebraic matroid.
Indeed, if $U \notin \Delta (F)$, then $U$ contains some spread, $\spr (f_i) \subset U.$ Since $f_i\in \mathcal{I}(X)$ is a nonzero polynomial vanishing on $X$, the coordinates indexed by $U$ satisfy a nontrivial algebraic relation on $X$, and hence $\pi_U$ is not a dominant map. Thus $U \notin \Delta(X)$.

Since $\Delta (F)$ and $\Delta (X)$ are both downward-closed, one can see that establishing the inclusion in~\Cref{eq:hard-inclusion} on the level of facets is sufficient to show that $F$ is a universal Gr\"{o}bner basis.
Thus, we obtain the following simplification of \Cref{thm:larson-huang}.

\begin{corollary}\label[corollary]{cor:thm-restated}
Let $X \subset \mathbb{C}^n$ be an irreducible closed subvariety of affine space, and let $F \subset \mathcal{I}(X)$ be a squarefree set of polynomials vanishing on $X.$ If every facet of the spread complex $\Delta (F) $ is a basis in the algebraic matroid $\Delta (X)$, then $F$ forms a universal Gr\"{o}bner basis of $\mathcal{I}(X).$ 
\end{corollary}

To apply this universal Gr\"{o}bner basis criterion, we will determine both the facets in the spread complex associated to the Segre-determinantal minors and the bases in the algebraic matroid of the Segre-determinantal locus.

As a warm-up, we describe how this criterion may be applied in the context of more classical determinantal ideals.

\begin{example}\label{ex:2x2-nonexample}
Suppose $F$ is the set of $2\times 2$ minors of a $3\times 3$ matrix of distinct indeterminates,
$$
P = \begin{bmatrix}
    p_{11} & p_{12} & p_{13}\\
    p_{21} & p_{22} & p_{23} \\
    p_{31} & p_{32} & p_{33}
\end{bmatrix},
$$
$$
F = \{ 
\det (P_{ij, k\ell })
\mid 
1\le i < j \le 3,
\, 
1\le k < \ell \le 3
\} .
$$
The spread complex $\Delta (F)$ contains facets of different sizes.
For example, 
$$U = \{ p_{13}, p_{23}, p_{31}, p_{32}, p_{33} \}  \in \Delta (F)$$ 
is a facet of size $5$, complementary to $\spr \left( \det (P_{12, 12}) \right)$, while 
$$
U' = \{ p_{12}, p_{13}, p_{21}, p_{23}, p_{31}, p_{32} \} \in \Delta (F) 
$$
is a facet of size $6$, obtained by deleting the diagonals of $P.$
Since all bases of a matroid must have the same size, these observations imply that $\Delta (F)$ is \emph{not} the algebraic matroid of the variety $X = \mathcal{V} (F) \subset \mathbb{C}^9$.
For the general $2\times 2$ minor ideal for a matrix of distinct indeterminate entries, 
there is a well-known construction of a
universal Gr\"{o}bner basis 
(see e.g.~\cite{villarreal2001monomial},~\cite[Ch.~5]{bruns2022determinants}), which consists of both the original minors and certain degree-$4$ binomials.
This construction may also be verified using \Cref{thm:larson-huang}---see~\cite{huang2024fine} for details.
\end{example}

\begin{example}\label{ex:maximal-example}
The maximal minors of an $m\times n$ matrix of distinct indeterminates $P$ are known to form a universal Gr\"{o}bner basis.
This result dates back to work of Bernstein, Sturmfels, and Zelevinsky~\cite{BZ93,SZ93}---see also~\cite[Ch.~5]{bruns2022determinants} for a more recent perspective.
\Cref{thm:ugb} below generalizes this result to the maximal minor ideals defining Segre-determinantal loci.
As a warm-up, we sketch a short proof of the Bernstein-Sturmfels-Zelevinsky result based on \Cref{thm:larson-huang}.

Assuming $m\ge n$, the spreads of maximal minors are simply all $n\times n$ submatrices of $P$. 
Thus, any facet $U$ in the spread complex must  be a union built from $n-1$ complete rows of $P$ and $m-n+1$ incomplete rows with a single entry deleted from each.
Given $\tilde{p}\in \mathbb{C}^{|U|},$ points in the fiber $P\in \pi_U^{-1} (\tilde{p})$ are solutions to a \emph{low-rank matrix completion problem}, meaning that $P$ is a matrix of rank less than $n$ with entries partially specified by $\widetilde{p}.$ 
For generic $\widetilde{p},$ the entries deleted to form $U$ can be set so that there is a unique completion $P$ of rank $n-1.$
This shows the map $\pi_U$ is dominant.
We conclude, using \Cref{cor:thm-restated}, that the maximal minors form a universal Gr\"{o}bner basis.
\end{example}

\begin{remark}\label{ugb-remark}
\Cref{ex:2x2-nonexample,ex:maximal-example} illustrate two cases where the algebraic matroid of the variety of bounded-rank matrices can be characterized as the spread complex of a set of squarefree polynomials.
Fairly little is known about the algebraic matroid of bounded-rank matrices in general: we refer to~\cite{tsakiris2024results,nicklasson2025determinantal,bernstein2017completion,brakensiek2026rigidity}, however, for further results on this subject.
Universal Gr\"{o}bner bases beyond the maximal and $2\times 2$ cases are also not well-understood.
\end{remark}

We now prove the main theorem of this section.

\begin{theorem}\label{thm:ugb}
For $n > M(m_\bullet ),$ the maximal minors of the matrix $\mathcal{N}_{m_\bullet , n} (p_\bullet)$  form a universal Gr\"{o}bner basis for the vanishing ideal of $\Seg_{m_\bullet , n}.$
\end{theorem}
\begin{proof}
  We apply~\Cref{cor:thm-restated} to \(\mathcal{V}_{\mathrm{aff}}(I_{m_\bullet,n})\), which is irreducible by~\Cref{thm:sdl-prime-cm}. Let \(F\) be the set of maximal minors of \(\mathcal{N}_{m_\bullet,n}(p_\bullet)\).  These minors are squarefree because each determinant term uses one squarefree matrix entry from each selected row, and different rows involve disjoint variables.

  By~\Cref{cor:thm-restated}, it is enough to show that every facet \(U\) of \(\Delta(F)\) is a basis of \(\Delta(\mathcal{V}_{\mathrm{aff}}(I_{m_\bullet,n}))\).  Since \(\mathcal{V}_{\mathrm{aff}}(I_{m_\bullet,n})\) is irreducible, the bases of this algebraic matroid all have cardinality
  \[
    \dim \mathcal{V}_{\mathrm{aff}}(I_{m_\bullet,n})
    =
    n(d(m_\bullet)+r)-\bigl(n-M(m_\bullet)\bigr)
    =
    M(m_\bullet)+n\bigl(d(m_\bullet)+r-1\bigr).
  \]

  A maximal minor using the rows indexed by \(K\subset [n]\), \(|K|=M(m_\bullet)+1\), has spread equal to the full set of variables in those rows.  No variable from another row can occur, and any selected-row variable \(p_{j,a}^{(k)}\) appears by choosing a column \(\alpha\) with \(\alpha^{(j)}=a\), expanding along row \(k\), and evaluating the corresponding cofactor on coordinate Segre rows indexed by \(\alpha'\ne\alpha\), where it becomes a nonzero permutation determinant.  Thus a facet \(U\) of \(\Delta(F)\) is exactly \(M(m_\bullet)\) complete rows together with all but one variable from each remaining row.  Since a fixed row contains
  \[
    \sum_{j=1}^r(m_j+1)=d(m_\bullet)+r
  \]
  variables, every facet \(U\) has cardinality
  \begin{align*}
    |U| &= M(m_\bullet)\bigl(d(m_\bullet)+r\bigr)
          +
          \bigl(n-M(m_\bullet)\bigr)\bigl(d(m_\bullet)+r-1\bigr)\\
        &=
          M(m_\bullet)+n\bigl(d(m_\bullet)+r-1\bigr)
          =
          \dim \mathcal{V}_{\mathrm{aff}}(I_{m_\bullet,n}).
  \end{align*}

  It remains to show that the coordinate projection \(\pi_U\) is dominant.  Relabel the rows so that \(U\) contains all variables in rows \(1,\ldots,M(m_\bullet)\), and write the unique deleted variable in row \(k>M(m_\bullet)\) as \(p_{j_k,a_k}^{(k)}\).  For the complete rows, let \(h_\alpha(p_\bullet)\) be the signed maximal cofactor polynomial obtained by deleting the column indexed by \(\alpha\).  By the preceding cofactor argument, each \(h_\alpha(p_\bullet)\) is nonzero.  The condition that an incomplete row \(k\) lie on the hyperplane determined by the complete rows is linear in \(p_{j_k,a_k}^{(k)}\), with coefficient polynomial
  \begin{equation}
    C_k(p_\bullet)=
    \sum_{\substack{\alpha\\ \alpha^{(j_k)}=a_k}}
    h_\alpha(p_\bullet)
    \prod_{\ell\ne j_k} p_{\ell,\alpha^{(\ell)}}^{(k)}.
  \end{equation}
  This coefficient is not identically zero, since after choosing any \(\beta\) with \(\beta^{(j_k)}=a_k\) and specializing the retained variables in row \(k\) to \(p_{\ell,\beta^{(\ell)}}^{(k)}=1\) for \(\ell\ne j_k\) and all others to \(0\), it reduces to the nonzero polynomial \(h_\beta(p_\bullet)\).  Therefore, on the dense open set where the polynomial evaluations of all \(h_\alpha(p_\bullet)\) and \(C_k(p_\bullet)\) are nonzero, the missing variables are uniquely determined.  For the completed point \(q_\bullet\), every row lies on the hyperplane \(H(q_\bullet)=[h_\alpha(q_\bullet)]_\alpha\).  The completed matrix has rank at most \(M(m_\bullet)\), so the image of \(\pi_U\) contains this dense open set.  Hence \(\pi_U\) is dominant, and every facet of \(\Delta(F)\) is a basis of \(\Delta(\mathcal{V}_{\mathrm{aff}}(I_{m_\bullet,n}))\).
\end{proof}

\section{Segre-Determinantal Loci in Computer Vision} \label{sec:sdl-in-cv}

We now return to the problem from the introduction, motivated by computer vision.
Our goal in this section is to introduce the image variety (\Cref{def:image-variety}) for generalized pinhole cameras and its relationship with the Segre-determinantal locus of the previous two sections.
In particular, we will see that $\Seg_{(1,1,1) , n}$ coincides with the image variety for three flatland cameras.

For an integer sequence $m_\bullet = (m_1, \ldots , m_r)$ as before, suppose now that we have a corresponding sequence $A_\bullet = (A_1, \ldots , A_r),$ where each $A_i : \P^N \dashrightarrow \P^{m_i}$ is a projective linear map.
Choosing standard coordinates on each projective space lets us identify each $A_i$ with a $(m_i +1 ) \times (N+1)$ matrix, which is well-defined up to scale.
Thus, we may regard the sequence $A_{\bullet }$ as a point

\begin{equation}\label{eq:matrix-space}
A_\bullet \in \P^{m_\bullet \times N} := \P \left( \CC^{(m_1+1) \times (N+1)} \right) \times \cdots \times \P \left( \CC^{(m_r+1) \times (N+1)} \right).
\end{equation}

(Here, $\P \left( \bullet \right)$ denotes the projectivization of a complex vector space.)

\begin{definition}\label{def:image-variety}
For any $n$, $N$, and $m_\bullet $ as above, we define the associated \emph{image variety} (c.f.~\cite{agarwal2024atlas}) to be the closed image of the rational map
\begin{align}
\Theta_{N, m_\bullet , n} : \left( \P^N \right)^{\times n} \times \P^{m_\bullet \times N} &\dashrightarrow \left( \P^{m_\bullet}\right)^{\times n} \nonumber  \\
(q_\bullet , A_\bullet ) &\mapsto \left((A_1 q_j,\ldots,A_rq_j) \mid 1\le j \le n \right)  \label{eq:general-image-map}.
\end{align}
\end{definition}

For $N=2$ and $m_\bullet = (1,1,1),$ this agrees with the flatlander imaging map in~\Cref{eq:flatland-image-map}.

\begin{proposition}\label{prop:image-contained-segre}
For $N< d(m_\bullet )$, we have
\[
\overline{\Image \left( \Theta_{N,m_\bullet , n}\right)}
\subset
\Seg_{m_\bullet , n}.
\]
\end{proposition}
Our proof of \Cref{prop:image-contained-segre} employs a rational map $\Psi_{N, m_\bullet }: \P^{m_\bullet \times N} \dashrightarrow \mathcal{T}_{N, m_\bullet } $ which maps the space of $r$-tuples of generalized cameras dominantly onto the associated variety of \emph{Grassmann tensors.}
Each point in this variety is a tuple of $r$th-order tensors, each considered up to scale,
\begin{equation}\label{eq:tensor-space}
\mathcal{T}_{N, m_\bullet } \subseteq \P \left(
\CC^{m_1+1} \otimes \cdots \otimes \CC^{m_r + 1}
\right)^{\times B (N, m_\bullet )},
\quad
\text{ w/ } B (N, m_\bullet ) = \binom{d(m_\bullet) + r }{N+r+1}.
\end{equation}
To define the map $\Psi_{N, m_\bullet },$ suppose first that we have $\Theta_{N, m_\bullet , n} (q_\bullet , A_\bullet ) = p_\bullet $ for some $p_\bullet \in \left(\P^{m_\bullet }\right)^{\times n}.$
Fixing homogeneous coordinates, this then implies that for all $1\le i \le r $ and $1\le j \le n$ there exists a nonzero scalar $\lambda_{ij}$ such that $A_i q_j = \lambda_{ij} p_{ij}.$
For fixed $j,$ we obtain a linear system
\begin{equation}\label{eq:matrix-focals}
\mathcal{M}_{N, m_\bullet , j} (A_\bullet, p_\bullet) \cdot
\begin{pmatrix}
q_j\\
-\lambda_{1j}\\
\vdots \\
-\lambda_{rj}
\end{pmatrix}
= \begin{pmatrix}
    0 \\
    \vdots \\
    0
\end{pmatrix},
\end{equation}
whose coefficient matrix is given in block form as
\begin{equation}\label{eq:matrix-focals-block}
\mathcal{M}_{N, m_\bullet , j} (A_\bullet, p_\bullet) =
\begin{pmatrix}
    A_1 & p_{1j}\\
    \vdots  & & \ddots  \\
    A_r & & & p_{rj}
\end{pmatrix} \in \CC^{\left( d(m_\bullet) + r  \right) \times  \left( N + r + 1\right)}.
\end{equation}
Since $N<d(m_\bullet)$, the maximal minors of the matrix in~\Cref{eq:matrix-focals} are determinants of $(N+r+1) \times (N+r+1)$ submatrices. These minors define multihomogeneous forms on $\P^{m_\bullet \times N} \times \P^{m_\bullet} $ which are linear in the $p_{ij}$ variables. Thus, every such minor can be written as
\begin{equation}\label{eq:grassmann-tensor-linear-forms}
\displaystyle\sum \Psi_{S, N, m_\bullet} (A_\bullet)_{i_1, \ldots , i_r} \cdot p_{1j}^{(i_1)} \cdots p_{rj}^{(i_r)},
\end{equation}
where $S$ is some set indexing $N+r+1$ rows of the matrix $\mathcal{M}_{N, m_\bullet , j} (A_\bullet, p_\bullet)$.
Finally, the rational map $\Psi_{N, m_\bullet }$ is defined component-wise by the projective equivalence classes of the coefficient tensors $\Psi_{S, N, m_\bullet } (A_\bullet)$ defining the linear forms in~\Cref{eq:grassmann-tensor-linear-forms},
\begin{align}\label{eq:grassmann-tensor-map}
\Psi_{N, m_\bullet} : \P^{m_\bullet \times N} &\dashrightarrow \mathcal{T}_{N, m_\bullet } \\
A_\bullet &\mapsto \bigg( [\Psi_{S, N, m_\bullet } (A_\bullet)] \mid  S  \subseteq [m_1+\ldots + m_r+r] \text{ w/ } \# S = N+r+1 \bigg) \nonumber
\end{align}

\begin{proof}[Proof of \Cref{prop:image-contained-segre}]
It suffices to prove the claim on the dense open subset of the domain of $\Theta_{N,m_\bullet,n}$ where $\Psi_{N,m_\bullet}$ is defined, since $\Seg_{m_\bullet,n}$ is closed.
Suppose $p_\bullet = (P_1, \ldots , P_n)$ lies in the image of this open subset, so that $p_\bullet = \Theta_{N, m_\bullet , n} (q_\bullet , A_\bullet)$ for some point $(q_\bullet , A_\bullet)$ with $\Psi_{N,m_\bullet}(A_\bullet)$ defined.
Choose a subset $S  \subseteq [ m_1+\ldots + m_r+r]$ of size $N+r+1 $ whose coefficient tensor $\Psi_{S,N,m_\bullet}(A_\bullet)$ is nonzero.
For every $1\le j\le n$, the linear system in~\Cref{eq:matrix-focals} has a nonzero solution, so the corresponding maximal minor in~\Cref{eq:grassmann-tensor-linear-forms} vanishes at $P_j$.
Thus the points $\sigma_{m_\bullet} (P_j)$ all lie on the hyperplane in $\P^{M(m_\bullet)}$ defined by $\Psi_{S,N,m_\bullet}(A_\bullet)$, and hence $p_\bullet \in \Seg_{m_\bullet , n}$.
\end{proof}

In general, the inclusion of \Cref{prop:image-contained-segre} is strict.
This may be seen by comparing the dimension of $\Seg_{m_\bullet , n}$ provided by \Cref{thm:sdl-irreducibility} with the \emph{expected dimension formula} for the image variety
\begin{align}
\mathbb{E} \dim \left(\Image \Theta_{N, m_\bullet, n} \right) &= \dim (\operatorname{Dom}  \Theta_{N, m_\bullet, n}) - \dim \operatorname{PGL}_{N+1} \nonumber \\
&=
N ( n + r) + (N+1) d(m_\bullet ) - (N^2 + 2N)
,\label{eq:expected-dim-image-variety}
\end{align}
where $n$ (depending on $N$ and $m_\bullet $ ) is sufficiently large.
This formula comes by observing that $\Theta_{N, m_\bullet , n}$ is invariant under the natural action of $H \in \operatorname{PGL}_{N+1}$, $$H \cdot (q_\bullet , A_\bullet ) = (Hq_1, \ldots , Hq_n, A_1 H^{-1}, \ldots , A_r H^{-1}).$$
For concrete values of $N$, $m_\bullet $, and sufficiently large $n,$ we may verify that the expected dimension in~\Cref{eq:expected-dim-image-variety} equals the true dimension of the image variety by computing tangent spaces.

Below, we consider several specific instances of the image varieties described above, as well as their corresponding Segre-determinantal loci. 
Each example gives pointers to the algebraic vision literature on which it is based, which may be useful for further study of image varieties.

First, we solve the problem stated in the introduction of our paper.

\begin{example}\label{ex:flatland-trifocal}
For $m_\bullet = (1,1,1)$, $N=2$, we obtain the image variety associated to \emph{three flatland cameras.}
In this case, the associated Grassmann tensor is the $2\times 2\times 2 $ trifocal tensor.
The study of this tensor in computer vision was pioneered in work of {\AA}str{\"o}m-Oskarsson~\cite{aastrom2000solutions} and Quan~\cite{quan2002two}.
In computer vision terminology, the flatlander trifocal tensor has ``no internal constraints'', a designation which is often justified by informal dimension-counting.

We now offer rigorous justification for such dimension-counting.
To ease notation, we write $\Psi = \Psi_{[6], 2, (1,1,1) }$ for the flatlander trifocal tensor map constructed in the proof of \Cref{prop:image-contained-segre}.

\begin{lemma}
  \label{lem:flatland-trifocal-dominant}
  The rational map
  $
    \Psi:\left(\P(\CC^{2\times 3})\right)^{\times 3}\dashrightarrow \left(\P^7\right)^*
  $
is dominant.
\end{lemma}
\begin{proof}
It suffices to prove that the corresponding map of affine varieties 
\[
\widehat{\Psi } : \left(\CC^{2\times 3}\right)^{\times 3}\dashrightarrow \mathbb{C}^8,
\]
is dominant, which follows upon observing that the \(8\times18\) Jacobian matrix $d \widehat{\Psi} (A_\bullet) $ at the point
  \[        
    A_\bullet =\left( 
    \begin{pmatrix}1 & 0 & 0 \\ 0 & 1 & 0\end{pmatrix}, \, 
    \begin{pmatrix}0&1 & 0\\ 0 & 0 & 1\end{pmatrix}, \, \begin{pmatrix}0&0&1\\ 1 & 0 & 0 \end{pmatrix}\right) ,
  \]
in an appropriate basis, contains a diagonal $8\times 8$ matrix with $\pm 1$ on the diagonal.
\end{proof}

\begin{theorem}\label{thm:flatland}
  For all $n\ge 1,$ we have
  \begin{equation}\label{eq:segre-equals-image-variety}
    \overline{\Image \Theta_{2,(1,1,1), n}}
    =
    \Seg_{(1,1,1), n} .
  \end{equation}
\end{theorem}
\begin{proof}
If $1\le n \le  M(m_\bullet ) = 7,$  then the equalities $\Seg_{(1,1,1), n} = \overline{\Image \Theta_{2,(1,1,1), n}}
    =\left( \mathbb{P}^1 \right)^{\times 3n}$ may be checked directly by computing tangent spaces.
Thus, we assume $n\ge 8.$ 

Since \(2<d(1,1,1)=3\), \Cref{prop:image-contained-segre} gives the inclusion
  \[
    \overline{\Image \Theta_{2,(1,1,1),n}}\subseteq \Seg_{(1,1,1),n}.
  \]
This inclusion and \Cref{thm:sdl-irreducibility} imply the inequality
\begin{equation}\label{eq:dim-inequality}
\dim \left( \overline{\Image \Theta_{2,(1,1,1),n}} \right) \le 
\dim \left( \Seg_{(1,1,1),n} \right) = 
(d(m_\bullet ) - 1) n + M(m_\bullet ) = 
2n + 7.
\end{equation}
Since both varieties $\overline{\Image \Theta_{2,(1,1,1),n}}$ and \(\Seg_{(1,1,1), n}\) are irreducible, the desired conclusion in~\Cref{eq:segre-equals-image-variety} will follow if we can prove
equality throughout~\Cref{eq:dim-inequality}.
Note that this agrees with the expected dimension in~\Cref{eq:expected-dim-image-variety},
\[
\mathbb{E} \dim \left(\Image \Theta_{2, (1,1,1), n} \right) =
2 ( n + 3) + (2+1) (1+1+1) - (2^2 + 2\cdot 2) = 
2n + 7.
\]
To prove the expected dimension is correct, consider the rational map
\begin{align*}
\rho : \overline{\Image \Theta_{2,(1,1,1),n}} &\dashrightarrow  \left(\mathbb{P}^7\right)^*\\
p_\bullet  &\mapsto \mathbb{P}\left(\ker \mathcal{N}_{(1,1,1), n} (p_\bullet )\right).
\end{align*}
This is well-defined: if $p_\bullet =(P_1, \ldots , P_n)\in \Image \Theta_{2,(1,1,1),n},$ then $\Theta_{2,(1,1,1),n} (q_\bullet , A_\bullet )= p_\bullet $ implies that the hyperplane in $\mathbb{P}^7$ defined by the trifocal tensor $\Psi (A_\bullet )$ contains $\sigma_{(1,1,1)}(P_1),\ldots , \sigma_{(1,1,1)}(P_n).$
Furthermore, $\rho $ is dominant; this follows from \Cref{lem:flatland-trifocal-dominant}.
Finally, observe that a generic fiber of $\rho $ has dimension $2n$: given a generic trifocal tensor $T=\Psi (A_1, A_2, A_3 )$ and generic points $p_{1j}, p_{2j}\in \mathbb{P}^1,$ there are uniquely determined $q_j \in \mathbb{P}^2$ with $A_1 q_j \sim p_{1j}$ and $A_2 q_j \sim p_{2j},$ and we may define $p_{3j} \sim A_3 q_j$ for all $j=1,\ldots , n$ to obtain $p_\bullet = (p_{11}, \ldots , p_{3n})\in \Image \Theta_{2,(1,1,1),n}$ with $p_\bullet \in \rho^{-1} (T )$.  Thus, 
$$
\dim \left( \overline{\Image \Theta_{2,(1,1,1),n}} \right) = 2n + 7
\quad 
\Rightarrow 
\quad 
\overline{\Image \Theta_{2,(1,1,1), n}}
    =
    \Seg_{(1,1,1), n}.
$$
\end{proof}
\end{example}

We now move from flatland to cases where the Grassmann tensors have nontrivial internal constraints.
\Cref{ex:fundamental-matrix,ex:trifocal-tensor,ex:quadrifocal-tensor} focus on classical pinhole cameras, i.e., linear projections $\mathbb{P}^3 \dashrightarrow \mathbb{P}^2$.
\Cref{ex:radial-quadrifocal-tensor} concerns projections $\mathbb{P}^3 \dashrightarrow \mathbb{P}^1,$ sometimes known as \emph{radial cameras}.

\begin{example}\label{ex:fundamental-matrix}
\cite{pratt2025segre}
For $N=3,$ $m_\bullet = (2,2)$, the inclusion of \Cref{prop:image-contained-segre} is strict:
\begin{equation}\label{eq:segre-strictly-includes-image-variety}
\overline{\Image \Theta_{3,(2,2), n}}
\subsetneq
\Seg_{(2,2), n}
\quad \text{for all } n \ge 8.
\end{equation}
In particular, when $n=8,$ \Cref{thm:sdl-irreducibility} implies that $\Seg_{(2,2), 8} = \left( \P^2 \times \P^2 \right)^{\times 8} = \left( \P^2 \right)^{\times 16}$.
The image variety in this case, however, is a hypersurface.
We briefly recall how the implicit equation of the hypersurface $\overline{\Image \Theta_{3,(2,2), 8}}$ may be derived from the Segre matrix $\mathcal{N}_{(2,2), 8}(p_\bullet)$ using a classical construction from computer vision.
For any point $p_\bullet = (P_1, \ldots , P_8) \in \overline{\Image \Theta_{3,(2,2), 8}} \subset \left( \mathbb{P}^2 \times \mathbb{P}^2 \right)^{\times 8}$,
\begin{align}\label{eq:fundamental-constraint}
   \det( \mathcal{M}_{3, (2,2) , j}(A_\bullet, p_\bullet)) = 
\det     \begin{pmatrix}
        A_1 & p_{1j} & \\
        A_2 & & p_{2j}
    \end{pmatrix}
    =
    p_{2j}^T F(A_1, A_2) p_{1j} = 0,
    \quad  j=1,\ldots , 8,
\end{align}
where the Grassmann tensor $F(A_1,A_2)$ is a $3\times 3$ rank-deficient matrix, known as the \emph{fundamental matrix} associated to the cameras $A_1, A_2.$ 
The constraints in~\Cref{eq:fundamental-constraint} have ``vectorized'' forms:
\[
 p_{2j}^T F(A_1, A_2) p_{1j} = 
 \left( p_{1j} \otimes p_{2j} \right)^T \textrm{vec} (F(A_1,A_2)),  
 \quad  j=1,\ldots , 8.
\]
In other words, the vectorized fundamental matrix must lie in the kernel of the $8\times 9$ Segre matrix.
Using Cramer's rule, we may write the kernel as a $9\times 1$ vector of polynomials in $p_\bullet $.
After suitably ``matricizing'' this $9\times 1$ vector into a $3\times 3$ matrix $F(p_\bullet ),$ we have
\begin{equation}\label{eq:8-point-constraint}
\overline{\Image \Theta_{3,(2,2), 8}}
= \mathcal{V} \left( \det F(p_\bullet )\right) \subset \left( \mathbb{P}^2\right)^{\times 16}.
\end{equation}
We refer to~\cite{hartley2003multiple} for background on the fundamental matrix, and~\cite{agarwal2017existence} for more details and applications of this particular construction.
For $n>8,$ our Segre-determinantal minors provide additional constraints
vanishing on $\overline{\Image \Theta_{3,(2,2),n}}$.
A natural question arises: for $n>8,$ do the $8$-point constraints in~\Cref{eq:8-point-constraint} and the Segre-determinantal minors suffice to give equations for $\overline{\Image \Theta_{3,(2,2),n}}$?
To make this more concrete, we may pose the problem of finding the following descriptions of the ideal of the image variety, in increasing order of difficulty: local equations, set-theoretic equations, scheme-theoretic equations, generators of the vanishing ideal, and/or (universal) Gr\"{o}bner bases.
Similar, more difficult, problems may be posed for the final three examples.
\end{example}

\begin{example}\label{ex:trifocal-tensor}
Consider now the case of three classical pinhole cameras: $N=3,$ $m_\bullet = (2,2,2)$.
Just like the case of two pinhole cameras in \Cref{ex:fundamental-matrix}, the inclusion $\overline{\Image \Theta_{3,(2,2,2),n}} \subset \Seg_{(2,2,2),n}$ is strict.
In particular, using \Cref{thm:sdl-irreducibility}, the Segre-determinantal locus is a strict subvariety of its ambient space only for $n \ge M(m_\bullet ) + 1 = 27$ points.
The image variety, however, is already a strict subvariety for $n \ge 7$ points.
The analogue of the fundamental matrix in this case is the $3\times 3 \times 3$ \emph{trifocal tensor}.
Whereas fundamental matrices are characterized by a single equation $\det F = 0,$ the ideal of trifocal tensors is very complicated; a complete description of this ideal is a result of Aholt and Oeding~\cite{aholt2014ideal}, which relies on sophisticated representation-theoretic machinery.
Alternatively, one may approach the ideal of $\overline{\Theta_{3,(2,2,2),n}}$ via the ideal of \emph{compatible fundamental matrix triples}, which was recently determined in~\cite{duff2026multiprojective}.
\end{example}

\begin{example}\label{ex:quadrifocal-tensor}
For four classical cameras---$N=3,$
$m_\bullet = (2,2,2,2)$---the Segre-determinantal locus is a strict subvariety of its ambient space only for $n \ge M(m_\bullet ) + 1 = 81$ points, whereas the image variety is a strict subvariety for $n \ge 6$ points.
The analogues of fundamental matrices and trifocal tensors are the so-called \emph{quadrifocal tensors}, whose implicit equations are only partially understood:~\cite{oeding2017quadrifocal} offers a census of equations in low degrees.
\end{example}

\begin{example}\label{ex:radial-quadrifocal-tensor}
Our final example concerns $N=3,$ $m_\bullet = (1,1,1,1)$. 
This is the setting of the so-called \emph{radial camera model}, concerning linear projections of the form $A:\mathbb{P}^3 \dashrightarrow \mathbb{P}^1$.
Such radial cameras have been used to model optical systems with radially symmetric lens distortion such as fisheye cameras. See~\cite{hruby-quadrifocal} for the implicit equations of radial quadrifocal tensors, and the references therein for more background.
The Segre-determinantal locus is a strict subvariety of its ambient space for $n \ge M(m_\bullet ) + 1 = 16$ points, whereas the image variety is a strict subvariety for $n \ge 14$.
\end{example}

\section*{Acknowledgements} 
We are grateful for the support and stimulating environment provided by the Fields Institute during the 2025 Workshop on Applications of Commutative Algebra, where this project was initiated.
We thank the workshop organizers, as well as Leo Jiang for helpful early discussions.

\bibliographystyle{amsplain}
\bibliography{references}
\end{document}